\documentclass[11pt,leqno]{amsart}
\usepackage{amsmath,amssymb,amsthm,amsfonts}
\usepackage[shortlabels]{enumitem}
\usepackage[english]{babel}
\usepackage{microtype}
\usepackage[colorlinks=false, pdfborder={0 0 0}]{hyperref}

\usepackage{marginnote} % add margin notes
\usepackage{bbm}
\usepackage[all]{xy}
\usepackage{tikz}
\usetikzlibrary{arrows}
\usetikzlibrary{shapes,decorations}
\usetikzlibrary{positioning}
\usepackage{tikz-cd} 
\usepackage[normalem]{ulem}
\usepackage{mathrsfs} % for \mathscr
	\definecolor{ceruleanblue}{rgb}{0.16, 0.32, 0.75}
\definecolor{airforceblue}{rgb}{0.36, 0.54, 0.66}
	\definecolor{aqua}{rgb}{0, 0.6, 0.6}
\definecolor{lightred}{rgb}{1, 0.8, 0.8}
\usepackage{bbm}

\DeclareMathOperator{\diam}{diam\,}
\DeclareMathOperator{\co}{co}

\DeclareMathOperator{\cconv}{\overline{\co}}

\newcommand{\ext}[1]{\operatorname{ext}\left(#1\right)}
\newcommand{\strexp}[1]{\operatorname{strexp}\left(#1\right)}
\newcommand{\wstrexp}[1]{w\operatorname{-strexp}\left(#1\right)}

\newcommand{\dent}[1]{\operatorname{dent}\left(#1\right)}

\renewcommand{\geq}{\geqslant}
\renewcommand{\leq}{\leqslant}

\renewcommand{\>}{\rangle}

\newcommand{\norm}[1]{\left\Vert#1\right\Vert}

\newtheorem{theorem}{Theorem}[section]
\newtheorem{lemma}[theorem]{Lemma}
\newtheorem{claim}[theorem]{Claim}
\newtheorem{proposition}[theorem]{Proposition}
\newtheorem{corollary}[theorem]{Corollary}
\newtheorem{question}[theorem]{Question}
\theoremstyle{definition}
\newtheorem{definition}[theorem]{Definition}
\newtheorem{example}[theorem]{Example}

\theoremstyle{remark}
\newtheorem{remark}[theorem]{Remark}
\numberwithin{equation}{section}
\def\fnote#1{\footnote}

\def\ignora#1{}
\def\n3#1{\left\vert  \! \left\vert \! \left\vert \, #1 \, \right\vert \!
  \right\vert \! \right\vert }

\newcommand{\pten}{\ensuremath{\widehat{\otimes}_\pi}}

\title{Extremal structure of projective tensor products}

\author[Garc\'ia-Lirola]{Luis C. Garc\'ia-Lirola}
\address[L. García-Lirola]{Departamento de Matemáticas, Universidad de Zaragoza, 50009, Zaragoza, Spain
\newline
	\href{http://orcid.org/0000-0001-9211-4475}{ORCID: \texttt{0000-0001-9211-4475} } }
 \email{\texttt{luiscarlos@unizar.es}}
\urladdr{\url{https://personal.unizar.es/luiscarlos/}}

\author[Grelier]{ Guillaume Grelier }
\address[G. Grelier]{Universidad de Murcia, Departamento de Matem\'aticas, Campus de Espinardo 30100 Murcia, Spain} \email{g.grelier@um.es}

\author[Mart\'inez-Cervantes]{Gonzalo Mart\'inez-Cervantes}
\address[G. Mart\'inez-Cervantes]{Universidad de Alicante, Departamento de Matem\'{a}ticas, Facultad de Ciencias, 03080 Alicante, Spain
	\newline
	\href{http://orcid.org/0000-0002-5927-5215}{ORCID: \texttt{0000-0002-5927-5215} } }	
\email{gonzalo.martinez@ua.es}

\author{Abraham Rueda Zoca}
\address[A. Rueda Zoca]{Universidad de Granada, Facultad de Ciencias.
Departamento de An\'{a}lisis Matem\'{a}tico, 18071-Granada
(Spain)
	\newline
	\href{https://orcid.org/0000-0003-0718-1353}{ORCID: \texttt{0000-0003-0718-1353} }}
\email{\texttt{abrahamrueda@ugr.es}}
\urladdr{\url{https://arzenglish.wordpress.com}}
\thanks{Research partially supported by Agencia Estatal de Investigación and EDRF/FEDER ``A way of making Europe" (MCIN/AEI/10.13039/501100011033) through grants PID2021-122126NB-C32 and PID2021-122126NB-C31 (Rueda Zoca), by Fundaci\'on S\'eneca Regi\'on de Murcia 20906/PI/18 (Garc\'ia-Lirola and Grelier). The research of García-Lirola was also supported by DGA project E48-20R.  The research of G.~Grelier was also supported by MICINN 2018 FPI fellowship with reference PRE2018-083703. The research of Abraham Rueda Zoca was also supported by Junta de Andaluc\'ia  Grants FQM-0185 and PY20\_00255.}

\keywords{Banach space; Projective tensor product; Preserved extreme point; Strongly exposed point}

\subjclass[2020]{
Primary 46B08, %% Ultraproduct techniques in Banach space theory
46B28  	%Spaces of operators; tensor products; approximation properties
}

\begin{document}

\maketitle

\markboth{GARC\'IA-LIROLA, MART\'INEZ-CERVANTES, GRELIER AND RUEDA ZOCA}{EXTREMAL STRUCTURE OF PROJECTIVE TENSOR PRODUCTS}

\begin{abstract}
We prove that, given two Banach spaces $X$ and $Y$ and bounded, closed convex sets $C\subseteq X$ and $D\subseteq Y$, if a nonzero element $z\in \cconv(C\otimes D)\subseteq X\pten Y$ is a preserved extreme point then $z=x_0\otimes y_0$ for some preserved extreme points $x_0\in C$ and $y_0\in D$, whenever $K(X,Y^*)$ separates points of $X \pten Y$ (in particular, whenever $X$ or $Y$ has the compact approximation property). Moreover, we prove that if $x_0\in C$ and $y_0\in D$ are weak-strongly exposed points then $x_0\otimes y_0$ is weak-strongly exposed in $\cconv(C\otimes D)$ whenever $x_0\otimes y_0$ has a neighbourhood system for the weak topology defined by compact operators. Furthermore, we find a Banach space $X$ isomorphic to $\ell_2$ with a weak-strongly exposed point $x_0\in B_X$ such that $x_0\otimes x_0$ is not a weak-strongly exposed point of the unit ball of $X\pten X$.
\end{abstract}

\section{Introduction}

One of the most celebrated and earlier results in Functional Analysis is Krein-Milman theorem. This result establishes that if $K$ is a compact convex subset of a locally convex space then $K=\cconv(\ext{K})$, where $\ext{K}$ denotes the set of extreme points of $K$ (see e.g. \cite[Theorem~3.22]{rud}). An example of application is to the unit ball of a dual Banach space $X^*$, where it %Krein-Milman theorem %reads as follows: given a Banach space $X$, then
yields that $B_{X^*}=\cconv^{w^*}(\ext{B_{X^*}})$. This result is of capital importance because, thanks to Hahn-Banach theorem, the structure of the geometry of a Banach space $X$ is determined by the dual unit ball $B_{X^*}$. Thus, Krein-Milman theorem tells us that the set $\ext{B_{X^*}}$ codifies %is enough to codify 
all the geometric information of the space. %of a given Banach space. 

%Since extreme points are useful to recover the whole structure of a compact convex set, 
The identification of the extreme points (and related notions as exposed, denting, or strongly exposed points) on particular classes of Banach spaces has attracted the attention of many researchers in functional analysis, especially in spaces where the definition of the norm is of high complexity, see e.g. \cite{CollinsRuess, ruessstegal, RuessStegallExp} for duals of spaces of compact operators, \cite{Kaminska} for Orlicz-Lorentz spaces, \cite{HM} for Kothe-Bochner spaces,  or, more recently, 
 \cite{aliaga, ag, APPP, gppr18, gpr18} for Lipschitz-free spaces. %or %\cite{ggr22} for ultrapowers of Banach space). 
In this note, we will focus on projective tensor products. 
%on another important construction of Banach spaces where the explicit formula of the norm encodes a high complexity: the projective tensor product. Given two Banach spaces $X$ and $Y$, its projective tensor product, %of $X$ and $Y$, 
Denoted by $X \pten Y$, the projective tensor product is the completion of the algebraic tensor product $X \otimes Y$ endowed with the norm
$$
\|z\|_{\pi} := \inf \left\{ \sum_{i=1}^n \|x_i\| \|y_i\|: z = \sum_{i=1}^n x_i \otimes y_i \right\},$$
where the infimum is taken over all such representations of $z$. Recall also that $B_{X\pten Y}=\cconv(B_X\otimes B_Y)$.

In the analysis of %direction of analysing 
the extremal structure of the projective tensor product we distinguish two lines.  The first one is the exhaustive analysis of the extreme points in duals of operators spaces done by Collins and Ruess \cite{CollinsRuess} and by Ruess and Stegall \cite{ruessstegal}. They established that, given two Banach spaces $X$ and $Y$, the extreme points of the dual unit ball of the $w^*$-to-$w$ continuous compact operators $K_{w^*}(X^*,Y)$ are the elements of the form $x^*\otimes y^*$, for $x^*\in \ext{B_{X^*}}$ and $y^*\in \ext{B_{Y^*}}$. As a consequence of a classical result of tensor product theory \cite[Theorem~5.33]{ryan}, if $X^*$ or $Y^*$ has the Radon-Nikodym property and $X^*$ or $Y^*$ has the approximation property then
\[ \ext{B_{X^*\pten Y^*}}=\ext{B_{X^*}}\otimes \ext{B_{Y^*}}.\]
%the extreme points of the unit ball of $X^*\pten Y^*$ are exactly those of the form $x^*\otimes y^*$ for $x^*\in \ext{B_{X^*}}$ and $y^*\in \ext{B_{Y^*}}.$ 

Little is known without the duality assumptions. Indeed, up to our knowledge, it is an open question whether every extreme point of $B_{X\pten Y}$ must be of the form $x\otimes y$ for $x\in B_X$ and $y\in B_Y$. The situation clarifies for the stronger notions of denting points and strongly exposed points. Ruess and Stegall proved in \cite{RuessStegallExp} that
\[ \strexp{B_{X\pten Y}} =\strexp{B_X}\otimes \strexp{B_Y}.\]
Furthermore, if $x^*$ strongly exposes $x$ in $B_X$ and $y^*$ strongly exposes $y$ in $B_Y$, then $x^*\otimes y^*$ strongly exposes $x\otimes y$ in $B_{X\pten Y}$. For denting points, D. Werner proved in \cite{werner} an analogous result in a more general framework: 
\[ \dent{\cconv(C\otimes D)} =\dent{C}\otimes \dent{D}\]
whenever $C\subset X$ and $D\subset Y$ are closed bounded and absolutely convex subsets.

%The second line of results in the direction of studying extremal structure are those from \cite{werner}, where the stronger notions of denting points and strongly exposed points are analysed. However, the results are presented in a quite more general framework. Indeed, the general framework of [Werner 87] is to analyse the denting points of $\cconv(C\otimes D)$ for absolutely convex bounded subsets $C\subseteq X$ and $D\subseteq Y$. It is proved that an element $z\in \cconv(C\otimes D)$ is a denting point (respectively strongly exposed) if, and only if, $z=x\otimes y$ for denting points (respectively strongly exposed points) $x\in C$ and $y\in D$. 

Motivated by the above results, in this note we study the notions of preserved extreme point and weak-strongly exposed point in projective tensor products. Recall that a point $x\in C$ is a \emph{preserved extreme point} of $C$ (also called \emph{weak*-extreme point}) if it is an extreme point of $\overline{C}^{w^*}\subset X^{**}$; this is a stronger notion than being extreme  but weaker than being denting  (see e.g. \cite{GMZ}).

With this notation in mind, one of the main results of the present paper is the following one.
%To begin with, in view of the results of \cite{ruessstegal, werner} it is natural to think that, in a projective tensor product, every extreme point of a set of the form $\cconv(C\otimes D)$ must be a basic tensor $x\otimes y$ for suitable extreme point $x\in C$ and $y\in D$. One of the main results in the paper establishes this result for preserved extreme points under the hypothesis of approximation property.

\begin{theorem}\label{condineceG}
Let $X$ and $Y$ be Banach spaces such that $K(X,Y^*)$ is separating for $X\pten Y$ (in particular, if the pair $(X,Y^*)$ has the CAP). Let $C\subseteq X$, $D\subseteq Y$ be bounded closed convex subsets. If $z$ is a preserved extreme point of $\cconv(C\otimes D)\subset X\pten Y$ then $z=x\otimes y$ for some $x\in C$ and $y\in D$. Moreover, if $z\neq 0$ then $x$ and $y$ are preserved extreme points of $C$ and $D$ respectively.\end{theorem}

As a particular case we get:

\begin{corollary}
Let $X$ and $Y$ be Banach spaces such that $K(X,Y^*)$ is separating for $X\pten Y$ (in particular, if the pair $(X,Y^*)$ has the CAP). If $z$ is a preserved extreme point of $B_{X\pten Y}$, then $z=x\otimes y$ where $x$ and $y$ are preserved extreme points of $B_X$ and $B_Y$ respectively.
\end{corollary}

Theorem~\ref{condineceG} points out that, in order to look for preserved extreme points in projective tensor products, we only have to pay attention to basic tensors. We do not know whether the converse holds. However, we will prove a kind of converse for $w$-strongly exposed points (we refer the reader to Section \ref{sec:notation} and Definition \ref{defi:compactneighsys} for the unexplained notions).

%that, for weak-strongly exposed points, we can establish a kind of converse when requiring extra assumptions. In this line, our main result is the following (we send the reader to Definition \ref{defi:compactneighsys} for the definition of compact neighbourhood system).

\begin{theorem}\label{theo:condisufi}
Let $X$ and $Y$ be Banach spaces such that $K(X,Y^*)$ is separating for $X\pten Y$ (in particular, if the pair $(X,Y^*)$ has the CAP). Let $C\subseteq X$ and $D\subseteq Y$ be symmetric, bounded closed convex subsets. Assume that $x\otimes y$ has a compact neighbourhood system for the weak topology in $\cconv(C\otimes D)\subset X\pten Y$. Then the following are equivalent:
\begin{itemize}
    \item[(i)] $x\otimes y$ is $w$-strongly exposed in $\cconv(C\otimes D)$.
    \item[(ii)]$x$ and $y$ are $w$-strongly exposed in $C$  and $D$, respectively. 
\end{itemize} 
In particular, if $C\otimes D$ is weakly compact, then \[\wstrexp{\cconv(C\otimes D)}=\wstrexp{C}\otimes \wstrexp{D}.\]
\end{theorem}

%It seems a bit artificial the assumption that $x\otimes y$ has a compact neighbourhood system in the above theorem. A natural question at this point is whether this assumption can be removed and if, in general, $x$ and $y$ weak-strongly exposed implies $x\otimes y$ weak-strongly exposed. Surprisingly or not, the answer is negative.
The assumption that $x\otimes y$ has a compact neighbourhood system in the above result might seem to be artificial but, surprisingly or not, it cannot be removed. 
Indeed, in Example~\ref{exam:l2renorming} we find a Banach space $X$ which is isomorphic to $\ell_2$ satisfying that there exists a $w$-strongly exposed point $x_0\in B_X$ and such that $x_0\otimes x_0$ is not a $w$-strongly exposed point of $B_{X\pten X}$.%, proving that the assumption on compact neighbourhood system in Theorem \ref{theo:condisufi} cannot be omitted.

\section{Notation and preliminary results}\label{sec:notation}

Throughout the paper we will only deal with real Banach spaces. Let $C$ be a bounded subset of a Banach space $X$. Given $x^*\in X^*$ and $\alpha>0$, we denote
\[ S(C, x^*, \alpha)=\{x\in C: x^*(x)>\sup x^*(C)-\alpha\}\]
the (open) slice of $C$ produced by $x^*$. 

We say that $x\in C$ is \textit{extreme} if the condition $x=\frac{y+z}{2}$ with $y,z\in C$ implies $y=z$. We write $x\in\ext{C}$.

A point $x\in C$ is said to be \textit{exposed} if there exists $x^*\in X^*$ such that $x^*(x)>x^*(y)$ for all $y\in C\setminus\{x\}$. We also say that $x^*$ exposes $x$ in $C$. A point $x\in C$ is said \textit{strongly exposed} (resp. \textit{$w$-strongly exposed}) if there exists $x^*\in X^*$ exposing $x$ and such that for all sequences $(x_n)_n\subset C$ such that $x^*(x_n)\xrightarrow[n]{}x^*(x)$, it follows that $x_n\xrightarrow[n]{}x$ (resp. $x_n\xrightarrow[n]{w}x$). Equivalently, the slices of $C$ produced by $x^*$ are a neighbourhood basis of $x$ for the norm (resp. weak) topology in $C$.\footnote{This notation should not be confused with the one in \cite{RuessStegallExp}, where a point in $B_{X^*}$ is called weak*-strongly exposed if it is strongly exposed by an element of $X$.} In this case, we write $x\in\strexp{C}$ (resp. $x\in\wstrexp{C}$).

A point $x\in C$ is a \emph{preserved extreme point} (or a \emph{$w^*$-extreme point}) if $x$ is an extreme point of $\overline{C}^{w^*}$. It can be proved that $x$ is a preserved extreme point if and only if the open slices containing $x$ form a basis for $x$ in the weak topology induced on $C$ (see \cite{LLT}). This characterization will be used twice in the proof of Theorem \ref{condineceG} without further mention.
Notice that, in particular, every $w$-strongly exposed point is a preserved extreme point.

\vspace{0.5cm}

Let us write here the following lemma, which we will use systematically throughout the text. A proof of this well-known result can be found in \cite[Lemma~7.21]{spearpreprint} (a preprint version of \cite{spear}).

\begin{lemma}\label{lemma:slicesconv}
Let $X$ be a Banach space. Let $A$ be a bounded subset of $X$ and write $C=\cconv(A)$. Let $R:=\sup_{x\in A}\Vert x\Vert$. Then, given $x^*\in X^*$, we have:
\begin{enumerate}
    \item $\sup_{x\in A} x^*(x)=\sup_{x\in C} x^*(x)$.
    \item Given $0<\varepsilon<\frac{1}{2}$ we have that
    $$S(C,x^*,\varepsilon^2)\subseteq \co(S(A,x^*,\varepsilon))+4R\varepsilon B_X.$$
\end{enumerate}
\end{lemma}

Given two Banach spaces $X, Y$, we denote $L(X,Y)$, $K(X, Y)$ and $F(X, Y)$ the spaces of linear, compact, and finite-rank operators, respectively. Recall that $(X\pten Y)^*=L(X, Y^*)$ isometrically. We refer the reader to \cite{ryan} for basic properties of tensor products. We denote by $\tau_c$ the topology of compact convergence in $L(X,Y)$, i.e.~the topology of uniform convergence on compact subsets of $X$. It is well known that $X$ has the approximation property if $\overline{F(X,X)}^{\tau_c}=L(X,X),$ whereas it has the compact approximation property if $\overline{K(X,X)}^{\tau_c}=L(X,X).$ The definition of the approximation property was extended to pairs of Banach spaces by E. Blonde in \cite{Bonde} as follows: The pair $(X,Y)$ is said to have the approximation property if $\overline{F(X,Y)}^{\tau_c}=L(X,Y)$. In a similar fashion we say that the pair $(X,Y)$ has the compact approximation property (CAP for short) if  $\overline{K(X,Y)}^{\tau_c}=L(X,Y)$ (see, for instance, \cite{DMCM}).
Notice that for any set $S \subset (X\pten Y)^*=L(X, Y^*)$, we have $\overline{S}^{\tau_c} \subset \overline{S}^{w^*}$. 
Since for a subspace $Z\subseteq X^*$ to separate points of $X$ is equivalent to the equality $\overline{Z}^{w^*}=X^*$, we do have the following lemma:

\begin{lemma}
Let $X$, $Y$ be two Banach spaces. If the pair $(X,Y^*)$ has the CAP, then $K(X,Y^*)$ separates points of $X\pten Y$.
\end{lemma}

We will make use of the previous lemma throughout the text without further mention. We finish this section by recalling that the pair $(X,Y^*)$ has the CAP if and only if the pair $(Y,X^*)$ has the CAP. It is immediate that if $X$ or $Y$ has the compact approximation property or the approximation property then the pair $(X,Y^*)$ has the CAP. As a consequence of \cite[Example~4.2]{Bonde}, for every $1\leq p < 2 < q < \infty$ and every subspaces $X \subset \ell_q$ and $Y\subset \ell_p$, the pair $(X,Y)$ has the CAP. Nevertheless, there are such subspaces $X$ and $Y$ failing the compact approximation property.

%\section{An auxiliary result}

In order to prove our results about extremal structure, we need the following topological result which is of independent interest.

\begin{theorem}\label{th:prodwclosedG}
Let $X$ and $Y$ be two Banach spaces such that such that $K(X,Y^*)$ is separating for $X\pten Y$ (in particular, if the pair $(X,Y^*)$ has the CAP). Let $C\subseteq X$ and $D\subseteq Y$ be two bounded subsets. Then the weak-closure of $C\otimes D$ in $X\pten Y$ is equal to $\overline{C}^w \otimes \overline{D}^w$, that is $\overline{C\otimes D}^w=\overline{C}^w \otimes \overline{D}^w$.
%If $C$ and $D$ are weakly closed then $C\otimes D\subseteq X\pten Y$ is weakly closed.
\end{theorem}

\begin{proof}
First, given $x\in \overline{C}^w$, we have that the operator $i\colon Y\to X\otimes Y$ given by $y\mapsto x\otimes y$ is continuous. Thus, it is also weak-to-weak continuous, so 
\[\{x\}\otimes \overline{D}^w =i(\overline{D}^w)\subseteq \overline{i(D)}^w=\overline{\{x\}\otimes D}^w.\]
This shows that $\overline{C}^w\otimes\overline{D}^w\subseteq \overline{\overline{C}^w\otimes D}^w$. Analogously, we get $\overline{C}^w\otimes D\subseteq \overline{C\otimes D}^w$ and so 
\[ \overline{C}^w\otimes \overline{D}^w \subseteq \overline{\overline{C\otimes D}^w}^w=\overline{C\otimes D}^w. \]

Now, given $z\in \overline{C\otimes D}^w$, take a net $(x_s\otimes y_s)$ in $C\otimes D$ such that $x_s\otimes y_s\rightarrow z$ weakly, and let us prove that $z=x\otimes y$ for certain $x\in \overline{C}^w$ and $y\in \overline{D}^w$. %Observe that the conclusion is immediate if $z=0$ because $0\in C$, so we can assume  $z\neq 0$.
We denote by $\overline{C}^{w^*}$ and $\overline{D}^{w^*}$ respectively the closure of $C$ and $D$ in the $w^*$ topology of $X^{**}$ and $Y^{**}$ respectively, which are $w^*$-compact because they are bounded.

Since $(x_s)_s\subset \overline{C}^{w^*}$ and $(y_s)_s\subset \overline{D}^{w^*}$ we can assume, up to taking a suitable subnet, that both $x_s\rightarrow x^{**}$ in the $w^*$-topology of $X^{**}$ and $y_s\rightarrow y^{**}$ in the $w^*$-topology of $Y^{**}$. 

\begin{claim}
For any compact operator $K\colon X \longrightarrow Y^*$, we have that 
$$ K(x_s)(y_s) \rightarrow K^{**}(x^{**})(y^{**}).$$
\end{claim}
\begin{proof}[Proof of the claim]
First, recall that $K^{**}\colon X^{**} \longrightarrow Y^{***}$ is a compact operator which satisfies $K^{**}(X^{**})\subseteq Y^{*}$. Fix $\varepsilon>0$. We claim that there exists $s_0$ such that $|K(x_s)(y_s)-K^{**}(x^{**})(y^{**})|<\varepsilon$ for every $s \geq s_0$.
Namely, we know that, since $K^{**}$ is compact, $K^{**}(x^{**})\in Y^*$ and $y_s\rightarrow y^{**}$ in the $w^*$-topology, there exists $s_0$ such that 
$$ \|K(x_s)-K^{**}(x^{**})\|<\varepsilon/(2R) \quad \mbox{ and }\quad |K^{**}(x^{**})(y_s)-K^{**}(x^{**})(y^{**})|<\varepsilon/2$$
for every $s \geq s_0$, where $R>0$ is such that $D \subset RB_{Y}$.
Then
\[\begin{split} |K(x_s)(y_s)-K^{**}(x^{**})(y^{**})|& \leq \|K(x_s)-K^{**}(x^{**})\|\|y_s\|\\ & +|K^{**}(x^{**})(y_s)-K^{**}(x^{**})(y^{**})|<\varepsilon \end{split}\]
for every $s \geq s_0$ as desired.
\end{proof}
 
%\footnote{Here I would like to make the rest of the proof shorter by the following argument, although I have to check if this is correct: Since $(x_s \otimes y_s)$ is weakly convergent to $z$, we do have $K(z)=K^{**}(x^{**})(y^{**})$ for every $K \in K(X,Y^*)$.
%Since $K(X,Y^*)$ is separating in $X \pten Y$, it is also separating in $(X \pten Y)^{**}$, so $z=x^{**} \otimes y^{**}$,
%which is only possible if $x^{**} \in X $ and $y^{**} \in Y$.
%}

Now, we claim we can assume $x^{**}\neq 0$ and $y^{**}\neq 0$. Indeed, if $x^{**}=0$ this would imply $0\in \overline{C}^{w}$. Moreover, since $z(K)=(0\otimes y^{**})(K)=0$ holds for every $K\in K(X,Y^*)$, which is separating for $X\pten Y$, we would get that $z=0$ so, taking any $y\in D$, we have $z=0\otimes y\in \overline{C}^w\otimes \overline{D}^{w}$ and the proof would be finished. Henceforth, we assume  $x^{**}\neq 0$ and $y^{**}\neq 0$ and, clearly, the above mentioned equality $z(K)=(x^{**}\otimes y^{**})(K)$ holding true for every $K\in K(X,Y^*)$ implies $z\neq 0$ too.

\begin{claim}\label{claim:xddinX}
$x^{**}\in X$ and $y^{**}\in Y$.
\end{claim}
\begin{proof}[Proof of the claim]  Let us prove that $x^{**}$ is $w^*$-continuous, being the proof for $y^{**}$ completely analogous. Take $y^*\in S_{Y^*}$ such that $y^{**}(y^*)\neq 0$. Now we have that
\[ x^{**}(x^*)=\frac{(x^{**}\otimes y^{**})(x^*\otimes y^*)}{y^{**}(y^*)}=\frac{(x^*\otimes y^*)(z)}{y^{**}(y^*)} \quad \forall x^*\in X^*.\]
Thus, to see that $x^{**}$ is weak*-continuous it suffices to show that $(x_s^*\otimes y^*)(z)\to (x^*\otimes y^*)(z)$ whenever $x^*_s\stackrel{w^*}{\to} x^*$. This is a consequence of the fact that the operator $X^*\to L(X,Y^*)$ given by $x^*\mapsto x^*\otimes y^*$ is $w^*$-to-$w^*$-continuous as being the adjoint of the operator $X\pten Y\to X$ given by $x\otimes y\mapsto y^*(y)x$.
\end{proof}

At this point we will save notation calling $x:=x^{**}\in X$ and $y:=y^{**}\in Y$. Now we have that $K(z)=K(x\otimes y)$ holds for every $K\in K(X,Y^*)$. Since $K(X,Y^*)$ is separating for $X\pten Y$, we deduce that $z=x\otimes y$. Moreover, observe that $x_s\rightarrow x$ in the weak topology of $X$. Since $x_s\in C$ for every $s$ we conclude that $x\in \overline{C}^w$. Analogously, $y\in D$, so $z=x\otimes y\in \overline{C}^w\otimes \overline{D}^w$, which finishes the proof.
\end{proof}

In spite of the fact that, under the approximation property, the tensor product of weakly closed sets is weakly closed, it is interesting to notice that if $C$ and $D$ are weakly compact, it does not follow that $C\otimes D$ is weakly compact in $X\pten Y$ (for instance, if we take $C=D=B_{\ell_2}$, then the sequence $(e_n\otimes e_n)_n$ is equivalent to the $\ell_1$-basis, c.f. e.g. \cite[Example~2.10]{ryan}).

\section{Main results}

The aim of this section is to present the proof of Theorems \ref{condineceG} and \ref{theo:condisufi}. We start with the proof of Theorem~\ref{condineceG}.

\begin{proof}[Proof of Theorem~\ref{condineceG}]
Since $z$ is a preserved extreme point, there is a neighbourhood basis $\{S_\alpha\}$ of $z$ for the weak-topology of $\cconv(C\otimes D)$ so that $S_\alpha$ is a slice for every $\alpha$. Now, since $S_\alpha$ is a slice of $\cconv(C\otimes D)$ we can find $x_\alpha\otimes y_\alpha\in S_\alpha\cap (C\otimes D)$ for every $\alpha$. Since $S_\alpha$ is a weak basis for the weak topology at $z$ we get that $z\in \overline{\{x_\alpha\otimes y_\alpha\}}^w$, and now Theorem~\ref{th:prodwclosedG} and the fact that $C$ and $D$ are weakly closed imply that $z=x\otimes y$ for certain $x\in C$ and $y\in D$. 

If $z\neq 0$ it is not difficult to prove that $x$ and $y$ are preserved extreme points of $C$ and $D$. Indeed, if $S(\cconv(C\otimes D),T_\alpha, \beta_\alpha)$ is a neighbourhood system of $x\otimes y$ for the weak topology in $\cconv(C\otimes D)$, then the family of slices $S_\alpha'$ defined as
$$S_\alpha':=\{x'\in C: T_\alpha(x')(y)>1-\beta_\alpha\}$$
is a neighbourhood system of $x$ for the weak topology of $X$.
\end{proof}

An immediate consequence of Theorem~\ref{condineceG} is the following corollary.

\begin{corollary}\label{condineceweakstrongG}
Let $X$ and $Y$ be Banach spaces such that $K(X,Y^*)$ is separating for $X\pten Y$ (in particular, if the pair $(X,Y^*)$ has the CAP). Let $C\subseteq X$ and $D\subseteq Y$ be convex bounded subsets. If $z$ is a $w$-strongly exposed point of $\cconv(C\otimes D)$ then $z=x\otimes y$ for some $x\in C$ and $y\in D$. Moreover, if $z\neq 0$ then $x$ and $y$ are $w$-strongly exposed points of $C$ and $D$ respectively.
\end{corollary}

\begin{proof}
Theorem~\ref{condineceG} provides points $x\in C$ and $y\in D$ such that $z=x\otimes y$. It remains to prove that $x$ and $y$ are $w$-strongly exposed points if $z\neq 0$. Let $T\in L(X,Y^*)$ $w$-strongly exposing $x\otimes y$ in $\cconv(C\otimes D)$, and define $f\in X^*$ by $f(v):=T(v)(y)$. It is immediate that $f$ $w$-strongly exposes $x$ in $C$. The argument for $y$ is analogous.  
\end{proof} 

%\begin{remark} Given two Banach spaces $X$ and $Y$, in \cite[Theorem~1]{werner} it is proved that if $C\subseteq X$ and $K\subseteq Y$ are bounded, closed and absolutely convex, then \[ \dent{\cconv(C\otimes D)} =\dent{C}\otimes \dent{D}.\] Example~\ref{exam:nonsymmetric} shows that this result is false if we remove the assumption of symmetry on $C$ and $D$. \end{remark}

%The conclusion is similar to the final part of the above proof. Indeed, if $x\otimes y$ is weak-strongly exposed by $T$, then the family $\{S(\cconv(C\otimes D),T,\alpha): \alpha>0\}$ is a neighbourhood basis of $x\otimes y$ for the weak topology in $\cconv(C\otimes D)$. Define $f\in X^*$ by $f(v):=T(v)(y)$ for every $v\in X$. It is immediate that  \end{proof}

Now we will analyse a possible converse for Corollary \ref{condineceweakstrongG}. The first result we find is the following.

\begin{proposition}\label{prop:condisufstrexp}
Let $X, Y$ be Banach spaces. Let $C\subseteq X$ and $D\subseteq Y$ be %\textcolor{blue}{weakly closed} and 
bounded and symmetric convex subsets. Let $x_0$ be a strongly exposed point of $C$ and $y_0$ be a $w$-strongly exposed point of $D$. Then $x_0\otimes y_0$ is a $w$-strongly exposed point of $\cconv(C\otimes D)$.
\end{proposition}%\mnote{\lcch{Esta prueba y la del teorema siguiente son muy parecidas. ¿Se puede obtener un resultado del otro? Notar que en realidad aquí estamos probando que $x_0\otimes y_0$ tiene cpt. neigh. system.}}
\begin{proof} 

By homogeneity, we may assume that $C\subseteq B_X$ and $D\subseteq B_Y$, so $R:=\sup_{z\in C\otimes D}\Vert z\Vert\leq 1$.
Assume that $x^*$  strongly exposes $x_0$ in $C$ and $y^*$ $w$-strongly exposes $y_0$ in $D$. We may also assume that $\sup x^*(C)=1=\sup y^*(D)$. Let us prove that $x^*\otimes y^*$ $w$-strongly exposes $x_0\otimes y_0$ in $\cconv(C\otimes D)$. To this end, pick $U:=\bigcap\limits_{i=1}^n S(\cconv(C\otimes D), T_i,\alpha_i)$ to be a relatively weakly open subset of $\cconv(C\otimes D)$ containing $x_0\otimes y_0$, with $\norm{T_i}=1$ for each $i$, and let us prove that $S(\cconv(C\otimes D), x^*\otimes y^*,\beta)\subseteq U$ for a suitable $\beta$. %Call $a:=\sup x^*(C)$ and $b:=\sup y^*(D)$. 
Notice that 
\[(x^*\otimes y^*)(x_0\otimes y_0)=x^*(x_0)y^*(y_0)=1=\sup_{z\in \cconv(C\otimes D)} (x^*\otimes y^*)(z)\]
(here we use Lemma~\ref{lemma:slicesconv} together with the fact that $C$ and $D$ are symmetric). %We can assume with no loss of generality that $\Vert T_i\Vert=1$ for every $1\leq i\leq n$.

Since $x_0\otimes y_0\in U$, we have $T_i(x_0)(y_0)>\sup_{z\in \cconv(C\otimes D)}T_i(z)-\alpha_i$ for every $1\leq i\leq n$. Thus we can find $\varepsilon_0>0$ so that $T_i(x_0)(y_0)>\sup_{z\in \cconv(C\otimes D)}T_i(z)-\alpha_i+\varepsilon_0$ for every $i$.

Since $x^*$ strongly exposes $x_0$, there is $\delta'>0$ such that $\diam(S(C,x^*,\delta'))<\frac{\varepsilon_0}{4}$. %Also, we can assume that $\delta'<\frac{\varepsilon_0}{16}$  because the slices are decreasing with $\delta'$. 
Moreover, notice that
$$y_0\in \bigcap\limits_{i=1}^n \left\{y\in D: T_i(x_0)(y)>\sup_{z\in \cconv(C\otimes D)}T_i(z)-\alpha_i+\varepsilon_0\right\},$$
which is a % non-empty
relatively weakly open subset of $D$ containing $y_0$. Since $y_0$ is weakly exposed by $y^*$ we can find $\delta''>0$ such that
$$S(D,y^*,\delta'')\subseteq \bigcap\limits_{i=1}^n \left\{y\in D: T_i(x_0)(y)>\sup_{z\in \cconv(C\otimes D)}T_i(z)-\alpha_i+\varepsilon_0\right\}.$$
Take $\delta:=\min\{\delta',\delta'', \varepsilon_0/4\}$. %, which satisfies that $\delta<\frac{\varepsilon}{16}$. 
Consider finally the slice $S(\cconv(C\otimes D),x^*\otimes y^*,\eta^2)$, where $0<\eta<\delta/4$. Let us prove that the previous slice is contained in $U$. To this end, notice that
\begin{align*}S(\cconv(C\otimes D),x^*\otimes y^*,\eta^2)&\subseteq \co(S(C\otimes D,x^*\otimes y^*,\eta))+4\eta B_{X\pten Y}\\
&\subseteq\co(S(C\otimes D,x^*\otimes y^*,\delta))+\delta B_{X\pten Y}=:A
\end{align*}
thanks to Lemma~\ref{lemma:slicesconv} and the choice of $\eta$. % Moreover, since $4R\eta<\delta$ and $\eta<\delta$ we have that
%$$\co(S(C\otimes D,x^*\otimes y^*,\eta))+4R\eta B_{X\pten Y}\subseteq \co(S(C\otimes D,x^*\otimes y^*,\delta))+\delta B_{X\pten Y}=:A,$$
So, it suffices to prove that $A\subseteq U$. To this end, pick $x\otimes y\in S(C\otimes D,x^*\otimes y^*,\delta)$. This means that $x^*(x)y^*(y)>1-\delta$, from where $x^*(x)>1-\delta\geq 1-\delta'$ and $y^*(y)>1-\delta\geq 1-\delta''$. By the definition of $\delta'$ and $\delta''$ we get that $\Vert x-x_0\Vert<\frac{\varepsilon_0}{4}$ and $T_i(x_0)(y)>\sup_{z\in \cconv(C\otimes D)}T_i(z)-\alpha_i+\varepsilon_0$ for every $i$. Hence
\[\begin{split}T_i(x)(y)\geq T_i(x_0)(y)-\Vert T_i\Vert\Vert x-x_0\Vert \Vert y\Vert& >\sup_{z\in \cconv(C\otimes D)}T_i(z)-\alpha_i+\varepsilon_0-\frac{\varepsilon_0}{4}\\& =\sup_{z\in \cconv(C\otimes D)}T_i(z)-\alpha_i+\frac{3\varepsilon_0}{4}.
\end{split}\]
An easy convexity argument implies that \[ T_i(u)>\sup_{z\in \cconv(C\otimes D)}T_i(z)-\alpha_i+\frac{3\varepsilon_0}{4} \quad \forall u\in \co(S(C\otimes D,x^*\otimes y^*,\delta)).\] Now, given $u\in A$, we have $u=v+w$ with $v\in \co(S(C\otimes D,x^*\otimes y^*,\delta))$ and $\Vert w\Vert\leq \delta\leq\varepsilon_0/4$. Then, 
\[\begin{split}
T_i(u)=T_i(v)+T_i(w) & \geq \sup_{z\in \cconv(C\otimes D)}T_i(z)-\alpha_i+\frac{3\varepsilon_0}{4}-\Vert w\Vert\\
& \geq \sup_{z\in \cconv(C\otimes D)}T_i(z)-\alpha_i+\frac{\varepsilon_0}{2}>\sup_{z\in \cconv(C\otimes D)}T_i(z)-\alpha_i
\end{split}\]
for each $i$. We conclude that $u\in U$, which proves that $A\subseteq U$ and the proof is finished.
\end{proof}

Note that in Proposition~3.3 we obtain a compact operator ($T\colon X\to Y^*$ given by $T(x)=x^*(x)y^*$) providing a neighbourhood basis for $x_0\otimes y_0$ for the weak topology in $\cconv(C\otimes D)$. This motivates to consider the following notion. 
%Before providing the proof of Theorem \ref{theo:condisufi} we need the following definition.

\begin{definition}\label{defi:compactneighsys}
Let $X$ and $Y$ be Banach spaces, and let $C\subseteq X$,  $D\subseteq Y$ be two subsets. We say that $x\otimes y\in C\otimes D$ \emph{has a compact neighbourhood system for the weak topology in $\cconv(C\otimes D)$} if, given any weakly open subset $U$ %of $\cconv(C\otimes D)$
containing $x_0\otimes y_0$, there are slices $S(\cconv(C\otimes D), T_i,\alpha_i)$ given by compact operators $T_i\in K(X, Y^*)$ such that 
\[ x_0\otimes y_0 \in \bigcap\limits_{i=1}^n S(\cconv(C\otimes D), T_i,\alpha_i)\subseteq U.\] %containing $x_0\otimes y_0$ so that $\bigcap\limits_{i=1}^n S(\cconv(C\otimes D), T_i,\alpha_i)\subseteq U$ and each $T_i$ is a compact operator.
\end{definition}

\begin{remark}\noindent
\begin{enumerate}[a)]
    \item The above definition has an easy interpretation in terms of nets: $x\otimes y$ has a compact neighbourhood system for the weak topology in $\cconv(C\otimes D)$ if, and only if, given a net $(z_\alpha)_\alpha \subset \cconv(C\otimes D)$, the condition $T(z_\alpha)\rightarrow T(x\otimes y)$ for every $T\in K(X,Y^*)$ implies $z_\alpha\rightarrow x\otimes y$ in the weak topology on $X\pten Y$. Equivalently, $x\otimes y$ is a  point of continuity of the formal identity \[I\colon (\cconv(C\otimes D),w)\longrightarrow (\cconv(C\otimes D),\sigma(X\pten Y, K(X,Y^*))).\]
    \item In the case that $K(X,Y^*)$ is separating for $X\pten Y$ (in particular, if the pair $(X,Y^*)$ has the CAP) and $C\otimes D$ is weakly compact, $\cconv(C\otimes D)$ is also weakly compact by  Krein-Smulyan theorem (see e.g. \cite[Theorem~II.~2.11]{du}) and $\sigma(X\pten Y, K(X,Y^*))$ is Hausdorff because $K(X,Y^*)$ is separating for $L(X, Y^*)$. Thus, the identity map $I$ above is a homeomorphism and so every $x\otimes y\in C\otimes D$ has a compact neighbourhood system.
    %then every point of $\cconv(C\otimes D)$ has a compact neighbourhood system for the weak topology. Indeed, the formal identity $I\colon (\cconv(C\otimes D),w)\longrightarrow (\cconv(C\otimes D),\sigma(X\pten Y, K(X,Y^*)))$ is bijective and continuous. Moreover, the domain is compact by Krein-Smulyan theorem (see e.g. \cite[Theorem II. 2.11]{du}) and the codomain is Hausdorff because $K(X,Y^*)$ is separating for $L(X, Y^*)$ thanks to the approximation property assumptions. Consequently, $I$ is a homeomorphism, in other words, the weak topology and $\sigma(X\pten Y, K(X,Y^*))$ agree on $\cconv(C\otimes D)$. %, from where the conclusion of the remark is immediate.
\end{enumerate}
\end{remark}
 
Now we are ready to present the proof of Theorem~\ref{theo:condisufi}.

\begin{proof}[Proof of Theorem~\ref{theo:condisufi}]
(i)$\Rightarrow$ (ii) follows from Corollary \ref{condineceweakstrongG}. % Since every $w$-strongly exposed point is also a preserved extreme point, we have that $z=x\otimes y$ for $x\in C$ and $y\in D$. Now, it is easy to check that if $T\in L(X, Y^*)$ $w$-strongly exposes $z$ in $\cconv(C\otimes D)$, then $y^*=T(x)\in Y^*$ $w$-strongly exposes $y$ in $D$ and $x^*=T(\cdot)(y)\in X^*$ $w$-strongly exposes $x$ in $C$.

(ii)$\Rightarrow$(i). Write $R:=\sup_{z\in C\otimes D}\Vert z\Vert$. Take $x_0^*\in X^*$ and $y_0^*\in Y^*$ $w$-strongly exposing $x_0$ and $y_0$ in $C$ and $D$, respectively, with  $x_0^*(x_0)=\sup x^*_0(C)=1$, and $y_0^*(y_0)=\sup y^*_0(D)=1$. Pick $U$ to be a weak neighbourhood of $x_0\otimes y_0$ in $\cconv(C\otimes D)$. By the assumption, we can assume that $U=\bigcap\limits_{i=1}^n S(\cconv(C\otimes D), T_i,\alpha_i)$ for certain compact operators $T_1,\ldots, T_n\colon X\to Y^*$. Furthermore, we can assume $\sup_{\cconv(C\otimes D)}T_i=1$ for every $i$. Let $\eta$ small enough so that $T_i(x_0\otimes y_0)>1-\alpha_i+\eta$ holds for every $1\leq i\leq n$. Moreover, observe that $x_0\in \bigcap\limits_{i=1}^n \{z\in C: T_i(z)(y_0)>1-\alpha_i+\eta\}$, which is a relatively weakly open subset of $C$. Since $x_0^*$ $w$-strongly exposes $x_0$ then there exists $\delta'>0$ %(which we can select as small as we wish, and we will declare latter how small we do want it) 
so that \[x_0\in S(C, x_0^*,\delta')\subseteq \bigcap\limits_{i=1}^n \{z\in C: T(z)(y_0)>1-\alpha_i+\eta\}.\]

Now, for every $1\leq i\leq n$, the set $T_i(S(C, x_0^*,\delta'))$ is a relatively compact subset of $Y^*$. Using the compactness condition on all the $T_i's$ we can find a finite set $x_1,\ldots, x_m\in S(C, x_0^*,\delta')$ so that $B(T_i(x_j),\frac{\eta}{8}), 1\leq j\leq m$, is a covering of $T_i(S(C, x_0^*,\delta'))$ for every $1\leq i\leq n$. Observe that $T_i(x_j)(y_0)>1-\alpha_i+\eta$ holds for every $1\leq i\leq n$ and $1\leq j\leq m$. Consequently, 
$$y_0\in\bigcap\limits_{i=1}^n\bigcap\limits_{j=1}^m\{y\in D: T_i(x_j)(y)>1-\alpha_i+\eta\}.$$
Since $y_0^*$ $w$-strongly exposes $y_0$ we can find $\delta''>0$ so that \[y_0\in S(B_Y,y_0^*,\delta'')\subseteq  \bigcap\limits_{i=1}^n\bigcap\limits_{j=1}^m\{y\in D: T_i(x_j)(y)>1-\alpha_i+\eta\}.\]

We claim now that 
\[ S(C, x^*_0, \delta')\otimes S(D, y^*_0, \delta'')\subset \bigcap_{i=1}^n S\left(\cconv(C\otimes D), T_i, \alpha_i-\frac{\eta}{2}\right).\]
Indeed, let $x\in S(C, x^*_0,\delta')$ and $y\in S(D, y^*_0, \delta'')$. 
%if $x\in C$ and $y\in D$ satisfy that $x_0^*(x)>\sup x^*(C)-\delta'$ and $y_0^*(y)>\sup y^*(D)-\delta''$ we conclude that $T_i(x)(y)>\sup_{w\in \cconv(C\otimes D)}T_i(w)-\alpha+\frac{\eta}{2}$ holds for every $1\leq i\leq q$. Indeed, since $x\in S(C,x_0^*,\alpha)$
We have, for every $i\in\{1,\ldots, n\}$, an index $j_i\in \{1,\ldots, m\}$ such that $\Vert T_i(x)-T_i(x_{j_i})\Vert<\frac{\eta}{2}$. On the other hand, since $S(D,y_0^*,\delta'')\subseteq  \bigcap\limits_{i=1}^n\bigcap\limits_{j=1}^m\{y\in D: T_i(x_j)(y)>1-\alpha_i+\eta\}$ we have that, for every $1\leq i\leq n$, $T_i(x_{j_i})(y)>1-\alpha_i+\eta$. Consequently
\[
T_i(x)(y)\geq T_i(x_{j_i})(y)-\Vert T_i(x_{j_i})-T_i(x)\Vert >1-\alpha_i+\eta-\frac{\eta}{2}
 =1-\alpha_i+\frac{\eta}{2}.
\]

Take $\delta:=\min\{\delta',\delta'', \frac{\eta}{8},\frac{\eta}{8R}\}$ and consider $S:=S(\cconv(C\otimes D),x_0^*\otimes y_0^*, \delta^2)$. Observe that $x_0\otimes y_0\in S$. Moreover, 
$$S\subseteq \co(S(C\otimes D,x_0^*\otimes y_0^*,\delta))+4R\delta B_{X\pten Y}$$
in virtue of Lemma~\ref{lemma:slicesconv}. Now, given $1\leq i\leq n$, since $1-\delta>\max\{1-\delta',1-\delta''\}$ we conclude that every element  $x\otimes y$ of $S(C\otimes D,x_0^*\otimes y_0^*,\delta)$ satisfies  $x_0^*(x)>1-\delta'$ and $y_0^*(y)>1-\delta''$, so $T_i(x)(y)>1-\alpha_i+\frac{\eta}{2}$. Since $T_i$ is a linear continuous functional on $X\pten Y$ we conclude that $T_i(z)\geq 1-\alpha_i+\frac{\eta}{2}$ holds for every $1\leq i\leq n$ and every $z\in \co(S(C\otimes D,x_0^*\otimes y_0^*,\delta))$. Henceforth, given $z\in S$ we can find $u\in \co(S(C\otimes D,x_0^*\otimes y_0^*,\delta))$ and $v\in B_{X\pten Y}$ so that $z=u+4R\delta v$. Now, given $1\leq i\leq q$ we get
$$T_i(z)=T_i(u)+4\delta R T_i(v)\geq 1-\alpha_i+\frac{\eta}{2}-4R\delta>1-\alpha_i,$$
%Now if we select $\delta'$ and $\delta''$ smaller than $\frac{\eta}{8 R}$ the previous number is strictly bigger than $1-\alpha_i$, 
from where we conclude that $z\in \bigcap\limits_{i=1}^q S(\cconv(C\otimes D), T_i,\alpha_i)=U$. This proves that $S\subseteq U$.

Summarising, we have proved that every relatively weakly open subset of $\cconv(C\otimes D)$ containing $x_0\otimes y_0$ actually contains a slice $S(\cconv(C\otimes D), x_0^*\otimes y_0^*,\alpha)$. Moreover, \[(x_0^*\otimes y_0^*)(x_0\otimes y_0)=\sup x_0^*(C)\sup y_0^*(D)=\sup (x_0^*\otimes y_0^*)(\cconv(C\otimes D)).\] %by Lemma \ref{lemma:slicesconv}.
Thus, $x_0^*\otimes y_0^*$ $w$-strongly exposes $x_0\otimes y_0$. %and the proof is finished.
\end{proof}

%Given two Banach spaces $X$ and $Y$, the main purpose of the paper is to study if given two points $x_0\in C$ and $y_0\in D$, if they are extreme (resp. preserved extreme, denting etc.) then so is $x_0\otimes y_0\in \overline{\co}(C\otimes D)$. In all our sufficient conditions we assume that $C$ and $D$ are symmetric because, if one of the sets is not symmetric, the result is false.

\begin{remark} The hypothesis of symmetry of $C$ and $D$ in  Theorem \ref{theo:condisufi} and Proposition \ref{prop:condisufstrexp} is needed. Indeed, let $C\subseteq X$ be any bounded subset with more than one point such that $0$ is a strongly exposed point of $C$, and let $D\subseteq Y$ be a bounded set with a strongly exposed point $y\in D$ satisfying that $-y\in D$ too. In spite of $0$ and $y$ being strongly exposed, the basic tensor $0\otimes y=0\in \overline{\co}(C\otimes D)$ is not an extreme point, since % Indeed, take $x\in C\setminus\{0\}$. We can write
$$0\otimes y=0=\frac{1}{2}(x\otimes y+x\otimes (-y)),$$
for any $x\in C\setminus\{0\}$. Similarly, the result in \cite[Theorem~1]{werner} about the denting points of $\cconv(C\otimes D)$ does not hold when $C$ or $D$ are non-symmetric. 
\end{remark} 
%which is a non-trivial convex combination of elements of elements of $\overline{\co}(C\otimes D)$, as desired.

% \begin{example}\label{exam:nonsymmetric}
%Let $X$ and $Y$ be two Banach spaces, let $C\subseteq X$ be a bounded subset with more than one point such that $0$ is a strongly exposed point of $C$, and let $D\subseteq Y$ be a bounded set with a strongly exposed point $y\in D$ satisfying that $-y\in D$ too. In spite of $0, y$ being strongly exposed, the basic tensor $0\otimes y=0\in \overline{\co}(C\otimes D)$ is not an extreme point. Indeed, take $x\in C\setminus\{0\}$. We can write $$0\otimes y=0=\frac{1}{2}(x\otimes y+x\otimes (-y)),$$ which is a non-trivial convex combination of elements of elements of $\overline{\co}(C\otimes D)$, as desired.  \end{example}

At this point one can wonder whether the assumption of the existence of the compact neighbourhood system can be removed in Theorem~\ref{theo:condisufi}. We will show that the answer is negative. Let us consider first the following example, where the set is not symmetric. %The following example shows that the answer is negative. % which shows that, to a certain extent, the assumptions of the result are sharp.

%Let us consider the following example.

\begin{example}\label{exam:conjuntista}
Consider $X=Y=\ell_2$, let $K:=\cconv\{e_n:n\in\mathbb N\}$ and $f:=\sum_{k=1}^\infty 2^{-k}e_k^*$. It is known that $0$ is $w$-strongly exposed by $f$ in $K$. Indeed, assume $(x_n)_{n=1}^\infty\subset K$ and $\lim_{n\to\infty}\<f,x_n\>=0$. Since $K\subset \ell_2$, each $x_n$ can be expressed as $x_n=\sum_{k=1}^\infty a_k ^n e_k$ with $a_k\geq 0$ and $\sum_{k=1}^\infty a_k^n \leq 1$. Therefore
\[ 0 = \lim_{n\to\infty} \<f, x_n\> = \lim_{n\to\infty} \sum_{k=1}^\infty 2^{-k} a_k^n \geq \lim_{n\to\infty} 2^{-k} a_k^n. \] 
This means that $\lim_{n\to\infty}\<e_k, x_n\> = 0$ for each $k\in \mathbb N$ and so $x_n\stackrel{w}{\to}0$. 
%(see e.g. \cite[Example 0.1.1]{lctesis})
. However, $f\otimes f$ does not weak-strongly exposes $0$ in $K\otimes K\subseteq \ell_2\pten \ell_2$. Even more, $0$ is not weakly strongly exposed in $K\otimes K$, i.e. there is no  bilinear form $B\colon \ell_2\times \ell_2\to \mathbb R$ such that $z_k\rightarrow 0$ weakly whenever $z_k\in K\otimes K$ satisfies that $B(z_k)\rightarrow 0$. 

In order to prove that, take a bilinear form $B$. For every $n\in\mathbb N$, the sequence $\{B(e_n,e_k)\}_{k\in\mathbb N}\rightarrow 0$ since $e_k$ is weakly null in $\ell_2$. Thus there is a subsequence $(e_{k_n})_n$ of $(e_n)_n$ satisfying that $B(e_n, e_{k_n})\rightarrow 0$. Observe that $e_n\otimes e_{k_n}\in K\otimes K$ for every $n\in\mathbb N$.

However, %in order to prove that
$(e_n\otimes e_{k_n})_n$ does not converge weakly to $0$ since it is isometrically equivalent to the $\ell_1$ basis; this follows by the same argument as for the diagonal $e_n\otimes e_n$, see e.g. \cite[Example~2.10]{ryan}. 
%, let us even prove that it is isometrically equivalent to the $\ell_1$ basis (with an argument  which is verbatim that of \cite[Example 2.10]{ryan}). To this end, pick scalars $\lambda_1,\ldots, \lambda_p\in\mathbb R$ and let us prove that
%$$\left \Vert \sum_{i=1}^p \lambda_i e_i\otimes e_{k_i} \right\Vert=\sum_{i=1}^p \vert \lambda_i\vert.$$

%To this end set $\alpha_i:=sign(\lambda_i)$ and define $B:\ell_2\times \ell_2\longrightarrow \mathbb R$ by
%$$B(x,y):=\sum_{i=1}^p \alpha_i x(i) y(k_i).$$
%We claim that $\Vert B\Vert\leq 1$. Indeed, given $x,y\in \ell_2$, we have by H\"older inequality that
%$$\vert B(x,y)\vert\leq \sum_{i=1}^p \vert x(i) y(k_i)\vert\leq \left(\sum_{i=1}^p x(i)^2\right)^\frac{1}{2}\left(\sum_{i=1}^p y(k_i)^2\right)^\frac{1}{2}\leq \Vert x\Vert \Vert y\Vert,$$
%so $\Vert B\Vert\leq 1$. Consequently $$\left \Vert \sum_{i=1}^p \lambda_i e_i\otimes e_{k_i} \right\Vert\geq B(\sum_{i=1}^p \lambda_i e_i\otimes e_{k_i})=\sum_{i=1}^p \vert \lambda_i\vert$$

\end{example}

\begin{remark}
The above argument also proves that $0$ is not weakly strongly exposed in $\cconv(K\otimes K)$. It also follows that $f\otimes f$ exposes $0$ in $\cconv(K\otimes K)$. Indeed, since $f$ exposes $0$ in $K$ we have that, given any $z\in K$, $f(z)=0$ if and only if $z=0$. Consequently, given $x\otimes y\in K\otimes K$ we have that $(f\otimes f)(x\otimes y)=f(x)f(y)=0$ implies that  either $x$ or $y$ equals $0$ and so $x\otimes y=0$. From this, and the fact that $\{e_n\otimes e_m\}_{n,m}$ is a Schauder basis for $\ell_2\pten \ell_2$, it follows that $(f\otimes f)(z)=0$ if and only if $z=0$ for every $z\in \cconv(K\otimes K)$.%\mnote{\lcch{Ok para $\co(K\otimes K)$. Por qué tb para el cierre?. OK Mañana incluyo el argumento de Abraham}}
\end{remark}

In order to find an example showing that Theorem~\ref{theo:condisufi} does not hold without the  assumption of the existence of a compact neighbourhood system, we will use the set $K$ from Example~\ref{exam:conjuntista} to construct an equivalent renorming $|\cdot|$ on $\ell_2$ satisfying that the new unit ball $B_{(\ell_2, |\cdot|)}$ has a $w$-strongly exposed point $x_0$ such that $x_0\otimes x_0$ is not $w$-strongly exposed in $B_{(\ell_2,|\cdot|)\pten (\ell_2,|\cdot|)}$.

%because the set $K$ described in Example \ref{exam:conjuntista} is not symmetric. We will give the desired example below where, using this set $K$, we will construct an equivalent renorming on $\ell_2$ satisfying that the new unit ball $B$ has a weak-strongly exposed point $x_0$ such that $x_0\otimes x_0$ is not weak-strongly exposed in $\ell_2\pten \ell_2$.

\begin{example}\label{exam:l2renorming}
Set an equivalent norm $|\cdot|$ on $\ell_2$ so that the new unit ball is $\cconv((K-e_1)\cup(-K+e_1)\cup \frac{1}{8} B_{\ell_2})$, where $K$ is the set described in Example~\ref{exam:conjuntista}. We claim that $-e_1\in B_{(\ell_2,|\cdot|)}$ is $w$-strongly exposed by $f:=\sum_{k=1}^\infty 2^{-k}e_k^*$. Observe that $f(-e_1)=-\frac{1}{2}$. Call $A:=K-e_1$ for simplicity. Since $f$ is linear, it is clear that 
$$\sup_{z\in B_{(\ell_2,|\cdot|)}} \vert f(z)\vert=\sup\limits_{z\in A\cup -A\cup \frac{1}{8} B_{\ell_2}} \vert f(z)\vert.$$
Observe that the above supremum equals $\sup_{z\in A}\vert f(z)\vert$ since $\vert f(z)\vert\leq \frac{1}{8}$ on $\frac{1}{8}B_{\ell_2}$ and by a symmetry argument.

On the other hand, given $z\in A$ we have $z=v-e_1$ for $v\in K$. Now $f(v-e_1)=-\frac{1}{2}+f(v)\leq -\frac{1}{2}$ since $f\geq 0$ on $K$. This proves that $\Vert f\Vert=1/2=\vert f(-e_1)\vert$. In order to see that $f$ $w$-strongly exposes $-e_1$ it remains to prove that if $f(z_n)\rightarrow -\frac{1}{2}$ with $(z_n)_n\subset B_{(\ell_2,|\cdot|)}$ then $z_n\rightarrow -e_1$ weakly. In order to do so, by a density argument, we can assume with no loss of generality that $z_n\in \co(A\cup -A\cup \frac{1}{8} B_{\ell_2})$. For every $n$, we can write
$$z_n=\alpha_n a_n+\beta_n(-a_n')+\gamma_n x_n$$
for $a_n,a_n'\in A, x_n\in \frac{1}{8} B_{\ell_2}$ and $\alpha_n,\beta_n,\gamma_n\in [0,1]$ with $\alpha_n+\beta_n+\gamma_n=1$ for every $n$.

Observe that 
$$f(z_n)=\alpha_n f(a_n)+\beta_n f(-a_n')+\gamma_n f(x_n)\geq \alpha_n f(a_n)+\gamma_n (-1/8)$$
since $f(-a_n)\geq 0$ for every $n\in\mathbb N$ and since $f(x_n)\leq 1/8$. Since $f(z_n)\rightarrow -1/2$ the unique possibility is that $\alpha_n\rightarrow 1$ (which implies $\beta_n\rightarrow 0$ and $\gamma_n\rightarrow 0$). Moreover, it is immediate that $f(a_n)\rightarrow -1/2$. Since $f$ $w$-strongly exposes $-e_1$ in $A$ we conclude that $a_n\rightarrow -e_1$ weakly, so $z_n\rightarrow -e_1$ weakly, as desired.

Finally, if we consider $e_1\otimes e_1$, we get that it is not $w$-strongly exposed in $B_{(\ell_2,|\cdot|)\pten( \ell_2,|\cdot|)}$. As in Example~\ref{exam:conjuntista}, given any bilinear and continuous form $B$, we can find a  strictly increasing sequence $(k_n)_n$ such that $B(e_n,e_{k_n})\rightarrow 0$, so $B(-e_1+e_n,-e_1+e_{k_n})\rightarrow B(e_1,e_1)$. However, if $k_1<k_2<\ldots$ we have that $\{e_n\otimes e_{k_n}\}$ is equivalent to the $\ell_1$ basis since $(\ell_2,|\cdot|)$ and $\ell_2$ are isomorphic (it follows for instance from \cite[Proposition~2.3]{ryan}) and therefore $((-e_1+e_n)\otimes(-e_1+e_{k_n}))_n$ is not weakly convergent to $e_1 \otimes e_1$.
%. Indeed, if we call $X:=(\ell_2,\vert\cdot\vert)$, where $\vert\cdot\vert$ is the norm whose unit ball is $B$, since $X$ is linearly isomorphic to $\ell_2$ it is immediate that $X\pten X$ is linearly isomorphic to $\ell_2\pten \ell_2$ (it follows for instance from \cite[Proposition 2.3]{ryan}). 
\end{example}

We end the paper with some open questions.

\begin{question} Is $x\otimes y$ a preserved extreme point of $B_{X\pten Y}$ whenever $x$ and $y$ are preserved extreme points of $B_X$ and $B_Y$?
%If $C\subseteq X$ and $D\subseteq Y$ are bounded %and symmetric subsets, given $x\in C$ and $y\in D$ which are preserved extreme points, is it true that $x\otimes y$ is a preserved extreme point of $\cconv(C\otimes D)$?
\end{question}

\begin{question} Is every (preserved) extreme point of $B_{X\pten Y}$ a basic tensor? 
\end{question}

\end{document}